\documentclass[a4paper]{article} 
\usepackage[12pt]{extsizes}
\usepackage[utf8]{inputenc}
\usepackage[english]{babel}
\usepackage{indentfirst}
\usepackage{misccorr}
\usepackage{graphicx}
\usepackage{amsmath, amssymb, amsthm}
\usepackage[T2A]{fontenc} 
\include{format}

\renewcommand{\leq}{\leqslant}
\renewcommand{\geq}{\geqslant}

\usepackage{geometry}
\def\abstract{
\begin{list}{}{
		\fontseries{m}
		\fontsize{10pt}{10}
		\selectfont
		\topsep=0pt
		\partopsep=0pt
		\parskip=0pt
		\itemsep=0pt
                \setlength{\rightmargin}{15mm}
                \setlength{\leftmargin}{\rightmargin}}
\item[]}                      
\def\endabstract{\end{list}\normalsize}

\newcommand{\keywords}[1]{%
	\begin{abstract}
	{\vspace{6pt}
	 {\it Keywords:~}#1}%
	 
	 2020 Mathematics Subject Classification 42A20
	\end{abstract}
}

\newtheorem*{theorem}{Theorem}

\begin{document} 

\begin{center}
\large\textbf{MASSIVE HELSON SETS}
\par\medskip

{Anastasiia Ianina}\\
\end{center}
\abstract
According to the Wik theorem, there exist massive Helson sets on the circle. In particular, they can be of Hausdorff dimension one. We extend Wik's result to the multidimensional case. 
\endabstract

\keywords{absolutely convergent Fourier series, Helson sets}

	\section{Introduction and statement of the result}

 We consider the space $A(\mathbb{T}^d)$ of all continuous functions on the torus $\mathbb{T}^d = \mathbb R^d/(2\pi \mathbb Z)^d$ whose Fourier series $$f(t) \sim \sum\limits_{k \in \mathbb{Z}^d} \widehat{f}(k) e^{i (k, t)}$$ converge absolutely. Here $\mathbb Z$ is the set of integers, $\mathbb R$ is the real line, $(\cdot, \cdot)$ is the usual inner product in $\mathbb R^d$, and \{$\widehat{f}(k), k \in \mathbb Z^d\}$ is the sequence of the Fourier coefficients of $f$: $$\widehat{f} (k) = \dfrac{1}{(2\pi)^d}\int\limits_{\mathbb{T}^d} f(t) e^{-i (k, t)}dt.$$ The space $A(\mathbb{T}^d)$ is a Banach algebra with respect to the natural norm $$||f||_{A(\mathbb T^d)} = \sum\limits_{k \in \mathbb{Z}^d} |\widehat{f}(k)|$$ and the usual multiplication of functions. It is often called the Wiener algebra.

	A compact subset $E$ of the torus $\mathbb{T}^d$ is called a Helson set if every continuous function on $E$ can be extended to $\mathbb{T}^d$ to a function from the Wiener algebra. In other words, $E$ is a Helson set if for every continuous function $f$ on $E$ there exists $\widetilde{f} \in A(\mathbb{T}^d),$ whose restriction $\widetilde{f}|_E$ to $E$ coincides with $f.$

  For basic facts on Helson sets, we refer the reader to \cite{GrahamHare, Kahane}. For example, it is well-known that a Helson set on the circle $\mathbb{T}$ can not contain arbitrary long arithmetic progressions (see \cite[Chaps. 3]{Kahane}) and, hence, every Helson set has zero Lebesgue measure. This result can be easily transfered to the multidimensional case.
 
Generally, at first glance Helson sets seem to be <<thin>>. On the other hand, Wik showed \cite{Wik} that for an arbitrary function $h$ satisfying $h(t)/t \to \infty$ as $t \to +0$ there exists a Helson set $E \subset \mathbb T$ of positive Hausdorff $h$-measure. In particular, there exists a Helson set $E \subset \mathbb T$ of Hausdorff dimension one. Actually, Wik's result is stronger: he proved the existence of a massive Kronecker set on $\mathbb T$ (for details, see the concluding Sec. 3, Remark 2). 

Throughout this paper, we denote by $h$ a nonnegative, continuous, and increasing function on $[0, +\infty)$ with $h(0) = 0$. We call any such $h$ admissible. Recall that the Hausdorff $h$-measure $\mu_h$ of a bounded set $E \subset \mathbb R^d$ is defined by  \begin{equation}
    \mu_h(E) = \lim\limits_{\xi \to 0} \left [ \inf\sum\limits_{\nu = 1}^{\infty}h(\text{diam}(U_\nu))\right ].
    \label{mes}
\end{equation} Here the infinum is taken over all countable coverings of $E$ by open balls $\{U_\nu\}$, whose diameters $\text{diam}(U_\nu)$ are at most $\xi$. For $\alpha > 0$ we denote by $h_{\alpha}$ the function $h(t) =  t^{\alpha}$ and write $\mu_{\alpha}$ instead of $\mu_{h_{\alpha}}$. The Hausdorff dimension $\dim_H(E)$ of a set $E$ is defined as follows: $$\dim_H(E) = \sup \{\alpha: \mu_{\alpha}(E) > 0\}.$$

In this paper, we extend the Wik theorem to the multidimensional case. Note that since a Helson set can not contain a Cartensian product of infinite sets (see Sec. 3, Remark 1), one can not obtain a massive Helson set in $\mathbb T^d$ by just taking the Cartesian product of massive one-dimensional Helson sets. Our result is the following theorem.
	\begin{theorem}
	Assume that an admissible function $h$ satisfies $$\dfrac{h(t)}{t^{d}} \to \infty \quad \text{as} \quad t \to +0.$$ Then there exists a Helson set $E \subset \mathbb{T}^d$ of positive Hausdorff $h$-measure. In particular, there exists a Helson set $E \subset \mathbb T^d$ with $\dim_H(E) = d$.
	\end{theorem}

     The Helson sets we will construct in this paper are Cantor-type sets, i.e., totally disconnected perfect sets. At the same time, one of the results due to J.-P. Kahane (see, e.g., \cite[Chap. 7,  Sec. 9]{Kahane}) implies that there exists a Helson set in $\mathbb T^2$ which is a continuous curve (see the details in \cite{Korner}); see also \cite{McGehee} and \cite{Muller} for further investigations in this direction. 
	In all known examples, continuous curves which are Helson sets have Hausdorff dimension one. It is natural to ask if there exist massive Helson sets which are continuous curves in $\mathbb T^2$, continuous surfaces in $\mathbb T^3$, etc. We will study this question elsewhere. 

The questions about the existence of massive Helson sets in $\mathbb T^d$ were posed by V. Lebedev (private communication).
	
	\section{Proof of the result} Throughout this paper, we use the following notation: $C(E)$ is the space of continuous functions on a compact set $E \subset \mathbb T^d$ with the usual sup-norm,
 $M(E)$ is the space of measures supported by $E$ with the norm $$||\mu||_{M(E)} = \sup\limits_{f \in C(E), \; ||f|| \leq 1}\left|\int\limits_{E} f(t) d\mu(t)\right|.$$ The Fourier coefficients of a measure $\mu$ are defined by $$\widehat{\mu}(k) = \int\limits_{\mathbb{T}^d} e^{-i (k, t)}d\mu(t), \qquad k \in \mathbb Z^d.$$
	It is well-known (see, e.g., \cite{Graham}), that a compact set $E$ is a Helson set if and only if every $\mu \in M(E)$ satisfies \begin{equation}
	     ||\mu||_{M(E)} \leq C \sup\limits_{k \in \mathbb Z^d} |\widehat\mu(k)|, 
	    \label{measure}
	    \end{equation}
	    where $C > 0$ is independent of $\mu$.
 Note that it is sufficient to verify this condition for real-valued measures.
	    
	    In a usual manner we identify the torus $\mathbb T^d$ with the cube $[0, 2\pi]^d$ and continuous functions on $\mathbb T^d$ with $2\pi$-periodic (with respect to each coordinate variable) continuous functions on $\mathbb R^d$.
	
 Throughout the proof by a cube we mean a closed $d$-dimensional cube in $[0, 2\pi]^d$ with axis-parallel sides. 
	
	The proof consists of three parts. In the first two parts we modify Wik's argument and construct a perfect nowhere dense Helson set on $\mathbb T^d$. In the third part  we will show that this set has a positive Hausdorff $h$-measure.

	\subsection{Construction}

Fix positive $\varepsilon < 1/4$ and an increasing sequence of positive integers $\{p_i\}_{i = 1}^{\infty}$, which we will specify later. We construct a set $E$ by induction so that $E = \bigcap\limits_{i = 0}^{\infty}E_i,$ where, for all $i = 0, 1, ...$,  we have $E_{i+1} \subset E_i$ and $E_i$ is the union of $N_i$ closed disjoint cubes $E_i^k, k = 1, ..., N_i$ with equal side length (say, $l_i$). We denote by $\rho_i$ the minimum of pairwise distances between $E_i^k, k = 1, ..., N_i$.

Let $E_0 = [0, 2\pi]^d$ and assume that $E_i$ is already constructed. We will construct the set $E_{i+1}$ so that $E_{i + 1} = \bigcap\limits_{j = 1}^{2^{N_i}} E_{i, j}$. Here, for all $j = 1, 2,...$,  we have $E_{i, j+1} \subset E_{i, j}$ and $E_{i, j}$ is the union of $N_{i, j}$ closed disjoint cubes $E_{i, j}^k, k = 1, ..., N_{i, j}$ with equal side length which we denote by $l_{i, j}$.

Let $F(E_i)$ be the collection of all continuous functions on $E_i$, which take values $\pm 1$. It is clear that $F(E_i)$ consists of $2^{N_i}$ functions. We denote them by $f_i^k, k = 1, ..., 2^{N_i}$.

Now we construct $E_{i, 1}$. Set $p_{i, 1} = p_{r}$, where \begin{equation}
	r = r(i) = \sum\limits_{v = 0}^{i-1}2^{N_v} + 1
	\label{index}
\end{equation} and consider the set
	\begin{equation}
	    E_{i, 1}^* = \bigcap\limits_{s = 1}^d\{x \in E_i: |f_i^1(x) - g_{i, 1}^s(x)| \leq \varepsilon\},
	    \label{metric}
	\end{equation} where $g_{i, 1}^s(x_1, ..., x_d) = \cos{(p_{i, 1}x_s)}$. 
	
	We require that $p_{i, 1}$ is sufficiently large so that the period of $\cos{(p_{i, 1}x_s)}$ is at least three times less than $l_i$. In this case, it is easy to see that $E_{i, 1}^*$ is the union of closed cubes (and, perhaps, their parts intersecting the boundary of $E_i$) with the side length $l_{i, 1} = C(\varepsilon)/p_{i, 1}$, where
 \begin{equation}
	    C(\varepsilon)  = 2\arccos(1-\varepsilon).
	    \label{konst}
	\end{equation}
 Indeed, in the case $d = 2$, it is clear that for any $k = 1, ..., N_i,$ the sets $\{x \in E_i^k: |f_i^1(x) - g_{i, 1}^1(x)| \leq \varepsilon\}$ and $\{x \in E_i^k: |f_i^1(x) - g_{i, 1}^2(x)| \leq \varepsilon\}$ are the unions of strips parallel to the axes $x_2$ и $x_1$, respectively. The width of each of the strips (if they do not intersect the boundary of $E_i^k$) is equal to the difference between the closest solutions of the equation $\cos(p_{i, 1} x_s) = 1 - \varepsilon$. In the case $d > 2$, the argument is similar and we leave it to the reader.
	
As we said above, the set $E_{i, 1}^*$ is the union of several closed cubes with the side length $l_{i, 1}$ and, perhaps, their parts, which we will not consider further. Denote by $n_{i, 1}^k$ the number of cubes occured in each of the set $E_{i}^k,$ $k = 1, ..., N_{i}$, and put $n_{i, 1} = \min\limits_{k}{n_{i, 1}^k}$. 
Now for all $k = 1, ..., N_{i}$ we arbitrarily choose $n_{i, 1}$ cubes which compose $E_{i,1}^*$ and contained in $E_{i}^k$ and let $E_{i, 1}$ be the union of these cubes. We write $E_{i, 1} = \bigcup_{k = 1}^{N_{i, 1}}E_{i, 1}^k$, where $N_{i, 1} = N_{i} \cdot n_{i, 1}$.
	
	To construct a set $E_{i, 2}$ we first form a set $E_{i, 2}^*$ by approximating (as in (\ref{metric})) the function  $f_i^2 \in F(E_i)$ by functions $g_{i, 2}^s(x_1, ..., x_d) = \cos{(p_{i, 2}x_s}),$ $s = 1, ..., d,$ where $p_{i, 2} := p_{r + 1}$, where $r = r(i)$ was defined in (\ref{index}). Namely, we set
	
\begin{equation*}
	E_{i, 2}^* = \bigcap\limits_{s = 1}^d\{x \in E_{i, 1}: |f_{i}^{2}(x) - g_{i, 2}^s(x)| \leq \varepsilon\}.
\end{equation*}
 Now we again require that $p_{i, 2}$ is sufficiently large so that the period of $\cos{(p_{i, 2}x_s)}$ is at least three times less than $l_{i, 1}$. 

Then for all $k = 1, ..., N_{i, 1}$ we arbitrarily choose $n_{i, 2}$ cubes which compose $E_{i, 2}^*$ and contained in $E_{i, 1}^k$. Here $n_{i, 2} = \min\limits_{k}{n_{i, 2}^k}$ and $n_{i, 2}^k$ is the number of cubes occured in each of the set $E_{i, 1}^k,$ $k = 1, ..., N_{i, 1}$. Let $E_{i, 2}$ be the union of all chosen cubes. Thus $E_{i, 2}$ consists of $N_{i, 2} = N_{i, 1} \cdot n_{i, 2}$ cubes with the side length $l_{i, 2} = C/p_{i, 2},$ where $C$ was define in (\ref{konst}).
	
	Similarly, approximating (as in (\ref{metric})) each of the functions $f_i^k$, $k = 3, ..., 2^{N_i}$, we get the sets $E_{i, j}, j = 3, ..., 2^{N_i}$, and put $E_{i + 1} = \bigcap\limits_{j = 1}^{2^{N_i}} E_{i, j}$. Thus the set $E_{i+1}$ is the union of $N_{i+1} = N_{i, 2^{N_i}}$ cubes and, according to the construction, $E_{i+1}$ has the following property: each function in $F(E_i)$ can be approximated on $E_{i+1}$ with accuracy $\varepsilon$ by at least one of the functions of the form $g_{i, j}^s(x_1, ..., x_d) = \cos(p_{i, j}x_s),$
	where $s \in \{1, ..., d\}, j \in \{1, ..., 2^{N_i}\}.$ This completes the induction.
 
	Finally, we have the sequence $E_1 \supset E_2 \supset ... \supset E_i \supset ...$. Let $E = \bigcap\limits_{i = 1}^{\infty} E_i$. In the next section we will show that $E$ is a Helson set.
	
	\subsection{$E$ is a Helson set}
	To prove that $E$ is a Helson set, it is suffice to show that (\ref{measure}) holds. Let $\mu \in M(E)$ be a real-valued measure. Fix $\delta > 0$. It is clear that there exists a continuous function $\phi$ on $\mathbb T^d$ such that $\sup\limits_{x \in \mathbb T^d}|\phi(x)| \leq 1$ and $$||\mu||_{M(E)} \leq \left|\int\limits_{\mathbb T^d} \phi(x) d\mu(x)\right| + \delta.$$ 
Choose integers $i$ and $j(i)$ large enough, so that for any fixed $k$ the inequality $|\phi(x) - \phi(y)| \leq \varepsilon$ holds for all $x, y \in E_{i, j(i)}^{k}$ (which is possible since $l_{i, j(i)} = C/p_{i, j(i)}$ and so $l_{i, j(i)} \to 0$ as $i \to \infty$). It is easy to see that there exists a continuous function $f$ on $E_{i, j(i)}$, which takes values $\pm 1, 0$, and such that $$|f(x) - \phi(x)| \leq 1/2 + \varepsilon  \qquad \text{for all} \; x \in E.$$ Indeed, consider three cases. The first one is when $k$ is such that $\min{\{\phi(x), \; x \in E_{i, j(i)}^k\}} \in [-1; -1/2]$, in this case we put $f(x) = -1$ for all $x \in E_{i, j(i)}^k$. The second one is when $\min{\{\phi(x), \; x \in E_{i, j(i)}^k\}} \in (-1/2; 1/2)$, then put $f(x) = 0$ for all $x \in E_{i, j(i)}^k$. Finally, the third one is when $\min{\{\phi(x), \; x \in E_{i, j(i)}^k\}} \in [1/2; 1]$, then put $f(x) = 1$ for all $x \in E_{i, j(i)}^k$. 
 
It is clear that there exist functions $f_1, f_2 \in F(E_{i, j(i)})$ such that $f = (f_1 + f_2)/2$. We will assume that $f_1$ and $f_2$ are defined on $\mathbb T^d$ and equal to zero outside $E_{i, j(i)}$. 
Note that, according to the construction of $E$, there exist positive integers $q_1, q_2$ and $s_1, s_2 \in \{1, ..., d\}$ such that $$|f_{\nu}(x) - g_{\nu}(x)| \leq \varepsilon \quad \text{for all} \; x \in E, \; \nu = 1, 2,$$ where $g_{\nu}(x) = g_{\nu}(x_1, ..., x_d) = \cos(q_{\nu} x_{s_{\nu}})$, $\nu = 1, 2$. Let $g= (g_1 + g_2)/2$, then \begin{gather*}||\mu||_{M(E)} \leq \left|\int\limits_{\mathbb T^d} \left(\phi(x) + f(x) - f(x) + g(x) - g(x) \right) d\mu(x)\right| + \delta \leq \\ \int\limits_{E} |\phi(x) -  f(x)| |d\mu(x)| + \int\limits_{E} |f(x) - g(x)| |d\mu(x)| + \left|\int\limits_{\mathbb T^d} g(x) d\mu(x)\right| + \delta \leq \\ \left(\dfrac{1}{2} + 2\varepsilon\right) ||\mu||_{M(E)} + \left|\int\limits_{\mathbb T^d} g(x) d\mu(x)\right| + \delta \leq \left(\dfrac{1}{2} + 2\varepsilon\right) ||\mu||_{M(E)} + \sup\limits_{k \in \mathbb Z^d} |\widehat\mu(k)| + \delta. \end{gather*} Consequently, $$\left(\dfrac{1}{2} - 2\varepsilon\right) ||\mu||_{M(E)} \leq \sup\limits_{k \in \mathbb Z^d} |\widehat\mu(k)| + \delta.$$ Since $\delta > 0$ was chosen arbitrarily and $\varepsilon \in (0; 1/4)$, the inequality (\ref{measure}) holds. Hence $E$ is a Helson set.
	
	\subsection{Estimation of $\mu_h(E)$}
In this section it will be convenient to have single indices for $E_{i, j}$ instead of double ones. Thus we renumber $E_{i, j}$ according to the natural order in which these sets were built during the construction of $E$. We denote them $E_i$, though it is a slight abuse of notation (since we encountered different sets $E_i$ in Section 2.1). So $E = \bigcap\limits_{i = 1}^{\infty}E_i$, for each $i$ the set $E_i$ is the union of $N_i$ cubes $E_i^k, k = 1, ..., N_i$, which we will call cubes of $i$-th rank. For all $k = 1, ..., N_{i-1}$ there are $n_{i}$ cubes of $i$-th rank in each of the cube $E_{i-1}^k$, so $N_i = N_{i - 1} \cdot n_i = \prod\limits_{j = 1}^in_j$. The side length of $E_i^k$ is equal to $l_i = C/p_i,$ where the constant $C$ was defined in (\ref{konst}). It is clear from the construction that there exists a constant $C_1 = C_1(\varepsilon)$ such that pairwise distances between the cubes $E_i^k$ are at least $\rho_i = C_1/p_i$.

Now we shall verify that for any admissible function $h$ the sequence $\{p_i\}_{i = 1}^{\infty}$ can be chosen so that $\mu_h(E) > 0.$ 
Without loss of generality, we may assume that $h(t)/t^d$ is monotonously decreasing. Otherwise, we replace $h$ by $$\widetilde{h}(t) = t^d \inf\limits_{\tau \leq t}\dfrac{h(\tau)}{\tau^d}.$$

Recall that it was required in the construction of $E$ that the sequence $\{p_i\}_{i =1}^{\infty}$  grows sufficiently fast, so that the period of $\cos(p_{i}x_s)$  is at least three times less than $l_{i-1}$. In other words, we require $$\dfrac{l_{i-1}}{2\pi/p_{i}} = \dfrac{C/p_{i-1}}{2\pi/p_{i}} \geq 3, \quad i \geq 2,$$ or, equivalently, \begin{equation}
	\begin{cases}
	\dfrac{p_{i}}{p_{i - 1}} \geq \dfrac{6\pi}{C}, \quad i \geq 2, \\
	p_1 \geq 3.
	\end{cases}
	\label{condition1}
	\end{equation}
	
	It is easy to see that, on $i$-th step of the construction, there are $\lfloor C_2 p_{i}^d / p_{i-1}^d \rfloor$ cubes (and their parts) appearing in each of $E_{i-1}^k$, where $C_2 = C_2(\varepsilon, d) = \left(\dfrac{C}{2\pi}\right)^d < 1$ (here $\lfloor u \rfloor$ denotes the integer part of $u$). Indeed, in the case $d = 2$, the number of strips parallel to the axes $x_1$ and $x_2$, which occured in the construction of $E_i^*$ (see the paragraph below (\ref{konst})), is at least $\dfrac{l_{i-1}}{2\pi/p_{i}}.$ Since $l_{i-1} = C/p_{i-1}$, there are $\lfloor C_2 p_{i}^2 / p_{i-1}^2 \rfloor$ cubes of $i$-th rank in each of cube of $(i-1)$-th rank. In the case $d > 2$, the argument is similar.
	
	Recall that in the construction of $E_{i}$, we excluded the parts of the cubes which might intersected the boundary of $E_{i-1}$ (that is, $(d-1)$-dimensional faces of the cubes $E_{i-1}^k$). By the above argument, there are at most $C_3 p_{i}^{d - 1}/p_{i - 1}^{d - 1}$ $(i \geq 2)$ such parts in each of $E_{i-1}^k$, where $C_3 = C_3(\varepsilon, d) = 2d \left(\dfrac{C}{2\pi}\right)^{d-1}$  (here the factor $2d$ corresponds to the number of $(d-1)$-dimensional faces of a $d$-dimensional cube).
	
	Thus we obtain
	$$n_{i} \geq C_2 \left(\dfrac{p_{i}}{p_{i - 1}}\right)^d - C_3 \left(\dfrac{p_{i}}{p_{i - 1}}\right)^{d-1}.$$ It is clear that there exists $0 < C_4 < 1$ such that\begin{equation}
 \begin{cases}
     C_4 \left(\dfrac{p_{i}}{p_{i - 1}}\right)^d \leq n_i \leq C_2 \left(\dfrac{p_{i}}{p_{i - 1}}\right)^d, \quad i \geq 2, \\
     C_4 p_1^d \leq n_1 \leq C_2 p_1^d.
	    \label{dimension}
 \end{cases}
	\end{equation}
	
	Now we construct an outer measure $\mu$ supported by $E$ as follows: we set $\mu(E_0) = 1$, and then $\mu(E_i^k) = 1/N_i$ for all $i$ and $k = 1, ..., N_i$.  Clearly, it is a measure on the ring generated by the sets $E_i^k$, $i = 1, 2, ...,$ and $k = 1, ...,  N_i$. By Carath\'eodory's extension theorem, it can be extended to an outer measure $\mu$ on the $\sigma$-algebra generated by this ring. It is easy to see that $\mu$ is supported by $E$ and $\mu(E) = 1$.

 Fix $\delta > 0$ and consider a covering of $E$ by open balls $U_{\nu}, \nu = 1, ..., \nu_0,$ with $\text{diam}(U_{\nu}) \leq \delta$ (it is suffices to consider finite coverings due to compactness of $E$). Let $U$ be one of these balls. Clearly, there exists $i$ with $\rho_i \leq \text{diam}(U) < \rho_{i-1}$. 
 
 Note that $U$ intersects at most $\left(\text{diam}(U)/\rho_i + 1\right)^d$  cubes $E_i^k$. Indeed, in the case $d = 2$, since $\text{diam}(U) < \rho_{i-1}$, the ball $U$ can intersect at most one cube of $(i-1)$-th rank, and hence at most $n_i$ cubes of $i$-th rank. Take the diameter of $U$ which is parallel to the $x_1$-axis, and project the cubes of $i$-th rank, intersecting $U$, on this diameter. These projections are disjoint closed intervals. Denote the number of them by $m$. We have $\text{diam}(U) \geq (m - 1)\rho_i,$ since $U$ contains $m-1$ intervals between projections, and each interval has length аt least $\rho_i$. Since $U$ intersects at most $m^2$ cubes of $i$-th rank, the required bound follows. In the case $d > 2$, the claim follows in a similar way.
 
Thus $U$ intersects at most $$
	\left(\dfrac{\text{diam}(U)}{\rho_i} + 1\right)^d \leq \left(\dfrac{2 \text{diam}(U)}{\rho_i}\right)^d$$ cubes $E_i^k$.
Since $\mu(E_i^k) = \dfrac{1}{N_i} = \dfrac{1}{n_1 \cdot ... \cdot n_i}$, by the above estimate we have
	    $$\mu(U) \leq \dfrac{1}{n_1 \cdot ... \cdot n_i} \left(\dfrac{2 \text{diam}(U)}{\rho_i}\right)^d.$$
Then, using (\ref{dimension}), we get $$\mu(U) \leq \dfrac{2^d}{(C_4)^i p_i^d} \left(\dfrac{\text{diam}(U)}{\rho_i}\right)^d.$$ 
     Since $h(t)/t^d$ is monotonously increasing as $t \to 0,$ we see that $$\dfrac{h(\rho_{i-1})}{\rho_{i-1}^d} \leq \dfrac{h(\text{diam}(U))}{\text{diam}(U)^d};$$ and thus $$\left(\dfrac{\text{diam}(U)}{\rho_i}\right)^d \leq \left(\dfrac{\rho_{i-1}}{\rho_i}\right)^d \dfrac{h(\text{diam}(U))}{h(\rho_{i-1})}.$$ Therefore, since $\rho_i = C_1/p_i$, we get \begin{equation}
     	\mu(U) \leq \dfrac{2^d}{(C_4)^i p_i^d} \left(\dfrac{\rho_{i-1}}{\rho_i}\right)^d \dfrac{h(\text{diam}(U))}{h(\rho_{i-1})} = \dfrac{2^d}{(C_4)^i} \dfrac{h(\text{diam}(U))}{p_{i-1}^d h(C_1/p_{i-1})}. \label{mesU}
     \end{equation} 
 Since $h(t)/t^d \to \infty$ as $t \to 0$, the sequence $\{p_i\}_{i = 1}^{\infty}$ can be chosen so that the inequality
 \begin{equation}
     \dfrac{2^d}{(C_4)^i p_{i-1}^d h(C_1/p_{i-1})} \leq 1
     \label{cond2}
 \end{equation} holds for all positive integers $i$.
 
Choosing a sequence $\{p_i\}_{i = 1}^{\infty}$ so that (\ref{condition1}) and (\ref{cond2}) hold, we obtain (see (\ref{mesU})) $$\mu(U) \leq h(\text{diam}(U))$$ for any ball $U$ whose diameter is at most $\delta$. Thus $$1 = \mu(E) \leq \mu \left(\bigcup\limits_{\nu = 1}^{\nu_0} U_{\nu}\right) \leq \sum\limits_{\nu = 1}^{\nu_0}\mu (U_\nu) \leq \sum\limits_{\nu = 1}^{\nu_0} h(\text{diam}(U_\nu))$$ and, hence, $\mu_h(E) > 0$.
	
		\section{Remarks}
	\begin{enumerate}
	    \item As we mentioned in the Introduction, a Helson set cannot contain a Cartesian product of infinite sets. For the sake of completeness we provide a short and simple proof of this fact. For simplicity, we consider the case $d=2$ (the proof in the general case is basically the same). 

	   Let $E \subset \mathbb T^2,$ $X_1 = \{x^1_1, ..., x^N_1\} \subset \mathbb T,$ $X_2 = \{x_2^1, ..., x^N_2\} \subset \mathbb T$, and $X_1 \times X_2 \subset E.$ Define the measure $$\mu = \sum\limits_{j, k = 1}^N u_{j, k} \delta_{j, k},$$ where $\delta_{j, k}$ is the Dirac delta function at the point $(x^j_1, x^k_2),$ and
    $$
    u_{j, k} = \dfrac{1}{\sqrt{N}} \exp\left(\dfrac{2\pi i}{N} j k\right).
    $$ 
    Clearly, the matrix $U = \{u_{j, k}\}_{j, k = 1}^N$ is unitary. Consider the vectors $a_{\lambda_1} = (e^{-i \lambda_1 x_1^1}, ..., e^{-i \lambda_1 x_1^N})$ and $a_{\lambda_2} = (e^{-i \lambda_2 x_2^1}, ..., e^{-i \lambda_2 x_2^N}),$ where $\lambda_1, \lambda_2 \in \mathbb Z.$ Obviously,
\begin{gather*}
	    \sup\limits_{\lambda \in \mathbb Z^2}|\widehat \mu(\lambda)| = \sup\limits_{\lambda_1, \lambda_2 \in \mathbb Z} \left | \sum\limits_{j, k = 1}^N u_{j, k} e^{-i \lambda_1 x^j_1} e^{-i \lambda_2 x^k_2} \right | = \sup\limits_{\lambda_1, \lambda_2 \in \mathbb Z} |(U a_{\lambda_1}, a_{\lambda_2})| \leq  \\ \leq ||U|| \cdot ||a_{\lambda_1}||_{2} \cdot ||a_{\lambda_2}||_{2} \leq ||U|| N = N,
	    \end{gather*}
	    where $||\cdot||_2$ is the standard norm in $\mathbb C^N$, and $||\cdot||$ stands for the norm of a matrix as of a linear map from $\mathbb C^N$ to $\mathbb C^N.$  
	    
	    Assuming that $E$ is a Helson set, we obtain (see (\ref{measure})) 
$$N^{3/2} = ||\mu||_{M(E)} \leq C \sup\limits_{\lambda \in \mathbb Z^2}|\widehat \mu(\lambda)| \leq C N,$$
	    where $C > 0$ does not depend on $N$, which is impossible if $N$ is large enough.  
	    
	   \item A compact set $E \subset \mathbb T^d$ is called a Kronecker set if the set  \{$e^{i(n, t)}, n \in \mathbb Z^d$\} is dense (with respect to the metrics of the space of continuous functions) in the set of complex-valued continuous functions on $E$, whose absolute value is equal to one. It is well-known that every Kronecker set is a Helson set (see, i.g., \cite[Chaps. 7, Sec. 7]{Kahane}). As we mentioned in the Introduction, the Wik theorem was in fact proved in a stronger form: for any nonnegative increasing continuous on $[0, +\infty)$ function $h$ with $h(t)/t \to \infty$ as $t \to 0,$ there exists a Kronecker set in $\mathbb T$ of positive Hausdorff $h$-measure. The question about the existence of a massive Kronecker set $\mathbb T^d$ remains open.

     Recall the question on the existence of massive Helson sets which are continuous curves in $\mathbb T^2$ (continuous surfaces in $\mathbb T^3$, etc.), stated in the Introduction. We note that a Kronecker set can not be a continuous curve since every Kronecker set is totally disconnected (see, e.g., Theorems 5.2.9 and 5.1.4 in \cite{Rudin}).
     \end{enumerate}
	
	\section{Acknowledgements} The author is grateful to Vladimir Lebedev for his help and attention to this work.


\begin{thebibliography}{3}
\bibitem{GrahamHare} Graham C. C., Hare K. E. Interpolation and Sidon sets for compact groups. -- Springer Science \& Business Media, 2013.
\bibitem{Graham}
Graham C. C., McGehee O. C. Essays in commutative harmonic analysis. – Springer Science and Business Media, 2012. -- Т. 238.
\bibitem{Kahane}
J.-P. Kahane, S\'{e}ries de Fourier absolument convergentes. -- Berlin: Springer, 1970.
\bibitem{Korner}  
T. W. K\"{o}rner, Kahane's Helson Curve. -- The Journal of Fourier Analysis and Applications, Kahane Special Issue, 1995, 325-346.
\bibitem{McGehee} O. C. McGehee and G. S. Woodward, Continuous manifolds in $\mathbb R^n$ that are sets of interpolation for the Fourier algebra. -- Ark, Mat., 20(1982), 169-199.
\bibitem{Muller} D. Müller, A continuous Helson surface in $\mathbb{R}^3$. -- Ann. Inst. Fourier, Grenoble, 34, 4, 1984, 135 - 150.
\bibitem{Rudin} W. Rudin, Fourier Analysis on Groups. -- Wiley-Interscience, 1990.

\bibitem{Wik}
I. Wik, Some examples of sets with linear independence, Arkiv for Matematik, Volume 5, Number 3-4 (1964), 207-214.

\end{thebibliography}
\end{document}